\documentclass{amsart}

\usepackage{amsmath}
\usepackage{amssymb}
\usepackage{amsthm}
\usepackage{graphicx}
\usepackage{enumitem}

\theoremstyle{plain}
\newtheorem{theorem}{Theorem}[section]
\newtheorem{prop}[theorem]{Proposition}
\newtheorem{lemma}[theorem]{Lemma}

\newtheorem{lemalph}{Lemma}

\theoremstyle{definition}
\newtheorem*{defn*}{Definition}

\begin{document}

\title[Compactifications of $\mathbb{R}$ with arc-like remainder]{Planarity of compactifications of $\mathbb{R}$ with arc-like remainder}
\author{Andrea Ammerlaan and Logan C. Hoehn}
\date{\today}

\address{Nipissing University, Department of Computer Science \& Mathematics, 100 College Drive, Box 5002, North Bay, Ontario, Canada, P1B 8L7}
\email{ajammerlaan879@my.nipissingu.ca}
\email{loganh@nipissingu.ca}

\thanks{This work was supported by NSERC grant RGPIN-2019-05998 and an NSERC CGS-M award.}


\begin{abstract}
We show that if $X$ is an arc-like continuum, then any continuum which is the union of $X$ and a ray $R$ such that $X \cap R = \emptyset$ and $\overline{R} \setminus R \subseteq X$ can be embedded in the plane $\mathbb{R}^2$.  Further, we prove that any compactification of a line with remainder $X$ is also embeddable in $\mathbb{R}^2$ -- answering a question of Sam B.\ Nadler from 1972.
\end{abstract}

\maketitle

\section{Introduction}
\label{sec:intro}

This paper concerns embeddings of certain kinds of continua in the plane $\mathbb{R}^2$, specifically continua which are the union of an arc-like continuum with a ray or line, which is disjoint from the arc-like continuum (definitions are provided below).  Our main result, Theorem~\ref{thm:line} below, is the answer to a question of Nadler from the 1972 paper \cite[p.229]{nadler1972}.  There, in light of his recent work with Quinn \cite{nadler-quinn1972, nadler-quinn1973} around embeddings of related continua in Euclidean spaces, Nadler asked whether each compactification of a line whose remainder is an arc-like continuum can be embedded in the plane.  Here, we prove the answer to this question is ``yes''.

Throughout, \emph{line} will mean a space homeomorphic to $\mathbb{R}$, \emph{ray} will mean a space homeomorphic to $[0,\infty)$, and an \emph{arc} is a space homeomorphic to a closed, bounded interval, e.g.\ $[-1,1]$.  A \emph{continuum} is a compact, connected, metric space.  A continuum $Y$ is a \emph{compactification} of a space $S$ if there exists a dense subset $S'$ of $Y$ homeomorphic to $S$.  The space $X = Y \setminus S'$ is called the \emph{remainder} in the compactification.  It is well-known that any continuum can be the remainder in a compactification of a line or ray (see e.g.\ \cite{nadler1972}) -- we consider compactifications for which the remainder is an arc-like continuum.

A continuum $X$ is \emph{arc-like} if it is homeomorphic to an inverse limit of arcs.  An \emph{inverse system} of continua is a sequence $\langle X_n,f_n \rangle_{n \geq 1}$ where, for each $n \geq 1$, $X_n$ is a continuum and $f_n \colon X_{n+1} \to X_n$ is a \emph{map}, i.e.\ a continuous function.  %The continua $X_n$ are called \emph{factor spaces} and the maps $f_n$ are called \emph{bonding maps}.
The \emph{inverse limit} of an inverse system $\langle X_n,f_n \rangle_{n \geq 1}$ is the space
\[ \varprojlim \langle X_n,f_n \rangle := \{\langle x_1,x_2,x_3, \ldots \rangle: f_n(x_{n+1}) = x_n \textrm{ for each } n \geq 1\} \subset \prod X_n .\]
In particular, if $X$ is arc-like, we can express $X$ in the form $X \approx \varprojlim \langle [-1,1],f_n \rangle$, and moreover, we may assume that for each $n \geq 1$, $f_n \colon [-1,1] \to [-1,1]$ is a nowhere constant, piecewise linear map \cite{brown1960}.  Bing \cite{bing1951} proved in 1951 that each arc-like continuum can be embedded in the plane.

A \emph{chain cover} of a continuum $X$ is a finite open cover $\mathcal{C} = \{U_1,\ldots,U_k\}$ of $X$, enumerated so that $U_i \cap U_j \neq \emptyset$ if and only if $|i-j| \leq 1$.  We often refer to the elements of a chain cover as \emph{links}.  Given $\varepsilon > 0$, we say that $\mathcal{C}$ is an $\varepsilon$-chain cover if $\mathrm{mesh}(\mathcal{C}) < \varepsilon$.  $X$ is \emph{chainable} if for each $\varepsilon > 0$ it admits an $\varepsilon$-chain cover.  It is well-known that a continuum $X$ is arc-like if and only if it is chainable.  We make use of this equivalence in Section~\ref{sec:main proofs}.

In \cite[Problem 4]{nadler1972}, Nadler asked, for any given $n$, which continua $X$ have the property (which he called ($\beta_n$)) that every compactification of a line with $X$ as remainder is embeddable in $\mathbb{R}^n$.  In his discussion around whether arc-like continua have property ($\beta_2$), he suggested that an answer to the following question, which came to be known as the Nadler-Quinn Problem, would be central: \textit{Given an arc-like continuum $X$ and a point $x \in X$, does there exist a plane embedding of $X$ for which $x$ is an accessible point?}  A point $x$ in a plane continuum $X \subset \mathbb{R}^2$ is \emph{accessible} if there exists an arc $A \subset \mathbb{R}^2$ such that $A \cap X = \{x\}$.

Recently, %%%the authors and Anu{\v{s}}i{\'c} answered the Nadler-Quinn Problem in the positive \cite{ammerlaan-anusic-hoehn2024}.  In the present paper, we obtain the following more general embedding theorem.
the Nadler-Quinn Problem was answered in the positive \cite{ammerlaan-anusic-hoehn2024}.  In the present paper, we obtain the following more general embedding theorem.

\begin{theorem}
\label{thm:ray}
Suppose $Y$ is a continuum of the form $Y = X \cup R$, where $X$ is an arc-like continuum, $R$ is a ray, $X \cap R = \emptyset$, and $\overline{R} \setminus R \subseteq X$.  Then $Y$ can be embedded in the plane.
\end{theorem}

In the special case that $\overline{R} \setminus R$ consists of a single point, this theorem coincides with the main result of \cite{ammerlaan-anusic-hoehn2024}.

Each continuum $Y$ of the form considered in Theorem~\ref{thm:ray} is simple-triod-like (possibly arc-like).  In general, a continuum is \emph{$S$-like}, for some continuum $S$, if it can be expressed as an inverse limit on spaces homeomorphic to $S$.  A \emph{simple triod} is a tree which is the union of three arcs which have one common endpoint and are otherwise pairwise disjoint.

The general question of which tree-like continua can be embedded in the plane is of significant interest due to its connections to other old and well-known problems in the field.  For example, it is relevant for the Plane Fixed Point Problem, which is one of the most important outstanding problems in continuum theory \cite[p.122]{bing1969}.  The problem asks if every non-separating plane continuum $X$ has the \emph{fixed point property}; i.e., is such that for any map $f \colon X \to X$, there exists a point which is fixed under $f$.  A plane continuum $X$ is \emph{non-separating} if $\mathbb{R}^2 \setminus X$ is connected; if $X$ is 1-dimensional, then it is non-separating if and only if it is tree-like \cite{manka2012}.  There have been several constructions of tree-like continua which do not have the fixed point property (see e.g.\ \cite{bellamy1980}, \cite{hernandezgutierrez-hoehn2018}, \cite{minc1992}, \cite{minc1996}, \cite{minc1999}, \cite{minc2000}, \cite{oversteegen-rogers1980}, \cite{oversteegen-rogers1982}).  If there exists a tree-like continuum which does not have the fixed point property and which can be embedded in the plane, then it would provide a negative answer to the Plane Fixed Point Problem.  In this context, we regard Theorem~\ref{thm:ray} as a contribution to the development of theory for discerning which tree-like continua are embeddedable in the plane.

In addition, Theorem~\ref{thm:ray} serves as a tool in our proof of the following theorem.

\begin{theorem}
\label{thm:line}
If $X$ is an arc-like continuum, then any compactification of a line with $X$ as remainder can be embedded in the plane.
\end{theorem}

In the notation of \cite{nadler1972}, Theorem~\ref{thm:line} states that every arc-like continuum has property ($\beta_2$).  It gives a class of circle- and noose-like continua which can be embedded in the plane.  A \emph{noose} is a space which is the union of a circle $C$ and an arc $A$ such that $C \cap A = \{e\}$, where $e$ is an endpoint of $A$.  See Figure~\ref{fig:circle noose like} for an example of a circle-like continuum and a noose-like continuum of the form considered in Theorem~\ref{thm:line}.

\begin{figure}
\begin{center}
\includegraphics[height=1in]{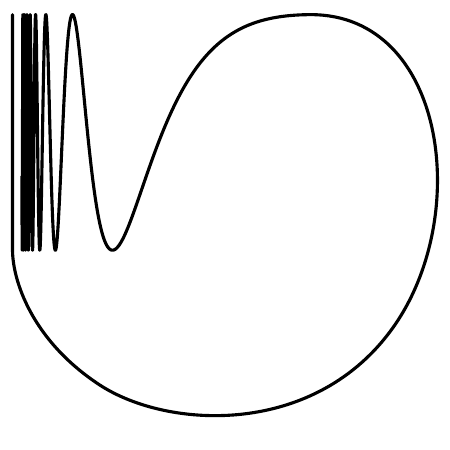}
\hspace{0.25in}
\includegraphics[height=1in]{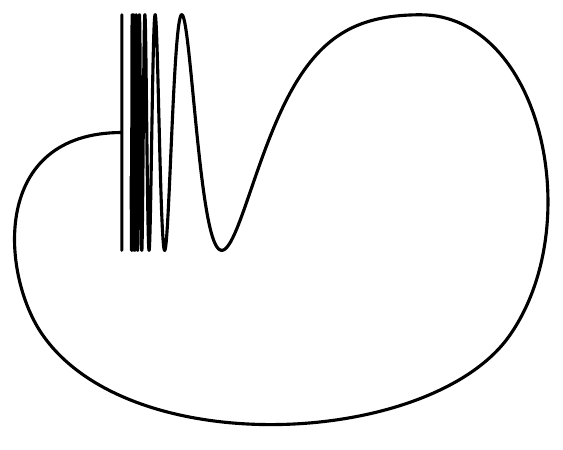}
\end{center}

\caption{Two variants of the Warsaw circle, each of which is a compactification of a line with an arc as the remainder.  The space on the left is circle-like, and the space on the right is noose-like.}
\label{fig:circle noose like}
\end{figure}

\section{Preliminaries}
\label{sec:prelim}

We first recall some notions and results from the papers \cite{ammerlaan-anusic-hoehn2023} and \cite{ammerlaan-anusic-hoehn2024} on the Nadler-Quinn problem.  Denote by $\mathbb{H}$ the half-plane $\{(x,y) \in \mathbb{R}^2 \colon x \geq 0\}$.

\begin{defn*}[Definition~8 of \cite{ammerlaan-anusic-hoehn2023})]
Let $f \colon [-1,1] \to [-1,1]$ be a map with $f(0) = 0$.  A \emph{radial departure} of $f$ is a pair $\langle x_1,x_2 \rangle$ such that $-1 \leq x_1 < 0 < x_2 \leq 1$ and either
\begin {enumerate}
\item $f((x_1,x_2)) = (f(x_1),f(x_2))$; or
\item $f((x_1,x_2)) = (f(x_2),f(x_1))$.
\end {enumerate}
We say a radial departure $\langle x_1,x_2 \rangle$ is \emph{positively oriented} (or simply \emph{positive}) if (1) holds and \emph{negatively oriented} (or simply \emph{negative}) if (2) holds.
\end{defn*}

The following result, Lemma~\ref{lem:no neg rad deps}, is extracted from the proof of the positive answer to the Nadler-Quinn Problem, Theorem~1.1 of \cite{ammerlaan-anusic-hoehn2024}.  There, it was sufficient to prove the lemma and then apply Lemma~\ref{lem:tuck} below.  For our purposes here, we will have use for both Lemmas~\ref{lem:no neg rad deps} and \ref{lem:tuck} rather than Theorem~1.1 of \cite{ammerlaan-anusic-hoehn2024} explicitly.

\begin{lemalph}
\label{lem:no neg rad deps}
Suppose $X = \varprojlim \langle [-1,1],f_n \rangle$ is an arc-like continuum and $x \in X$, where $x = \langle x_1,x_2,\ldots \rangle$ is such that $x_n = \pm 1$ for only finitely many $n$.  Then there exists a homeomorphism $h \colon X \to \varprojlim \langle [-1,1],f_n' \rangle$ such that each map $f_n'$ is piecewise linear, satisfies $f_n'(0) = 0$, and has no negative radial departures, and $h(x) = \langle 0,0,\ldots \rangle$.
\end{lemalph}

\begin{lemalph}[Lemma~13 of \cite{ammerlaan-anusic-hoehn2023}]
\label{lem:tuck}
Suppose $f \colon [-1,1] \to [-1,1]$ is a map with $f(0) = 0$, all of whose radial departures have the same orientation.  Then for any $\varepsilon > 0$ there exists an embedding $\Phi$ of $[-1,1]$ into $\mathbb{H}$ such that $\Phi(0) = (0,0)$, $\Phi([-1,1]) \cap \partial \mathbb{H} = \{(0,0)\}$, and $\|\Phi(x) - (0,f(x))\| < \varepsilon$ for all $x \in [-1,1]$.
\end{lemalph}

\section{Proofs of the main theorems}
\label{sec:main proofs}

We begin with two preliminary results.

\begin{lemma}
\label{lem:chain}
Suppose $Y$ is a continuum of the form $Y = X \cup R$, where $X$ is an arc-like continuum, $R$ is a ray, $X \cap R = \emptyset$, and $\overline{R} \setminus R \subseteq X$.  Let $\varepsilon > 0$, and suppose $\mathcal{C} = \{U_1,\ldots,U_k\}$ is an $\varepsilon$-chain cover of $X$ by open subsets of $Y$ such that $U_1 \cap \overline{R} \setminus R \neq \emptyset$.  Then there exists an $\varepsilon$-chain cover of $Y$.
\end{lemma}

\begin{proof}
Let $h$ be a homeomorphism from $[0,\infty)$ onto $R$.  Since $U_1 \cap \overline{R} \setminus R$ is nonempty, there exists $t > 0$ such that $h(t) \in U_1$ and $h([t,\infty)) \subset \bigcup \mathcal{C}$.  Let $U_1'$ be the open set $U_1 \setminus h([0,t])$.  Now, choose $t' > t$ such that $h([t,t']) \subset U_1$.  Then $\mathcal{C}_X := \{U_1'\} \cup \{U_i \setminus h([0,t']): 1 < i \leq k\}$ is a chain cover of $X$.  By covering $[0,t] \subset \mathbb{R}$ by a chain of sufficiently small open intervals in $[0,\infty)$ and taking their images under $h$, we may construct an $\varepsilon$-chain cover $\mathcal{C}_R = \{V_1,\ldots,V_m\}$ of $h([0,t])$, such that $V_j \cap \bigcup \mathcal{C}_X \neq \emptyset$ if and only if $j = 1$, and $V_1$ is disjoint from all links in $\mathcal{C}_X$ other than $U_1'$.  Then, $\mathcal{C}_R \cup \mathcal{C}_X$ is an $\varepsilon$-chain cover of $Y$.
\end{proof}

The next result is a stronger variant of Theorem~1.1 of \cite{ammerlaan-anusic-hoehn2024}, in which we take additional care to produce an embedding of an arc-like continuum $X$ in the half-plane $\mathbb{H}$ which can be covered by chains of round disks (i.e., balls in $\mathbb{R}^2$) of arbitrarily small mesh.  We remark that the assumption that $x_n = \pm 1$ for only finitely many $n$ in Proposition~\ref{prop:plane} could be removed, but we include it to keep the proof clean; the case where $x_n = \pm 1$ for infinitely many $n$ (which implies $x$ is an endpoint of $X$ -- see the proof of Theorem~\ref{thm:ray} below) will not be needed in the application of Proposition~\ref{prop:plane} below.

\begin{prop}
\label{prop:plane}
Let $X = \varprojlim \langle [-1,1],f_n \rangle$ be an arc-like continuum, let $x = \langle x_1,x_2,\ldots \rangle \in X$, and suppose that $x_n = \pm 1$ for only finitely many $n$.  Then there exists an embedding $e \colon X \to \mathbb{H}$ with $e(x) = (0,0)$ and such that for any $\varepsilon > 0$, there exists an $\varepsilon$-chain cover $\mathcal{C}$ of $e(X)$ satisfying:
\begin{enumerate}
\item each $U \in \mathcal{C}$ is a round disk in the plane; and
\item there exists a unique element $L \in \mathcal{C}$ such that $(0,0) \in L$.
\end{enumerate}
\end{prop}

\begin{proof}
In light of the assumption that only finitely many $x_n$ are endpoints of $[-1,1]$, we can apply Lemma~\ref{lem:no neg rad deps}.  As a result, we may assume that for each $n \geq 1$, $f_n$ is a piecewise linear map with $f_n(0) = 0$ and no negative radial departures, and $x_n = 0$.

To construct the desired embedding, we appeal to the Anderson-Choquet Embedding Theorem \cite{anderson-choquet1959}.  It suffices to show that, for any $n \geq 1$, if $\Omega_n$ is any embedding of $[-1,1]$ into $\mathbb{H}$ such that $\Omega_n(0) = (0,0)$, then for any chain cover $\mathcal{C}_n$ of $\Omega_n([-1,1])$ and any $\varepsilon_n > 0$, there exists a piecewise linear embedding $\Omega_{n+1} \colon [-1,1] \to \mathbb{H} \cap \bigcup\mathcal{C}_n$ such that
\begin{itemize}
\item $\Omega_{n+1}(0) = (0,0)$;
\item $\|\Omega_{n+1}(p) - \Omega_n(f_n(p))\| < \varepsilon_n$ for all $p \in [-1,1]$; and
\item there exists a chain cover $\mathcal{C}_{n+1}$ of $\Omega_{n+1}([-1,1])$ which satisfies properties (1) and (2), and such that $\overline{\bigcup \mathcal{C}_{n+1}} \subset \bigcup \mathcal{C}_n$.
\end{itemize}
According to the Anderson-Choquet Embedding Theorem, provided the values $\varepsilon_n$ are chosen to converge to $0$ quickly enough and such embeddings $\Omega_n$ are given, the limit of the images $\Omega_n([-1,1]) \subset \mathbb{R}^2$ (in the Vietoris topology) is homeomorphic to $X$.  Moreover, the condition $\overline{\bigcup \mathcal{C}_{n+1}} \subset \bigcup \mathcal{C}_n$ in the third bullet above ensures that this limit is covered by each of the chains $\mathcal{C}_n$.

To construct this embedding $\Omega_{n+1}$, we partially follow the proof of Proposition~7.2 of \cite{ammerlaan-anusic-hoehn2023}.

Let $A$ denote the straight segment in $\mathbb{R}^2$ from $(0,0)$ to $(-1,0)$ and let $I$ denote the straight segment from $(0,-1)$ to $(0,1)$.  Let $H \colon \mathbb{R}^2 \to \mathbb{R}^2$ be a homeomorphism such that $H(\Omega_n(p)) = (0,p)$ for each $p \in [-1,1]$, and $H {\restriction}_A = \mathrm{id}_A$.  For some small $\varepsilon^* >0$, apply Lemma~\ref{lem:tuck} to construct an embedding $\Phi \colon [-1,1] \to \mathbb{H}$ such that $\Phi(0) = (0,0)$, $\Phi([-1,1]) \cap \partial \mathbb{H} = \{(0,0)\}$ and $\|\Phi(p) - (0,f_n(p))\| < \varepsilon^*$ for all $p \in [-1,1]$.  Let $\Omega_{n+1} = H^{-1} \circ \Phi$.  Then $\Omega_{n+1}(0) = (0,0)$ and, by uniform continuity of $H$ in a neighbourhood of $A \cup I$, for sufficiently small $\varepsilon^*$, $\Omega_{n+1}$ is an embedding of $[-1,1]$ into $\mathbb{H}$ with $\|\Omega_{n+1}(p) - \Omega_n(f_n(p))\| < \varepsilon_n$ for all $p \in [-1,1]$.  By approximating if necessary, we may additionally ensure that $\Omega_{n+1}$ is piecewise linear.  Further, we may assume that $\varepsilon_n$ is small enough so that $\Omega_{n+1}([-1,1]) \subset \bigcup \mathcal{C}_n$.

Since $\Omega_{n+1}([-1,1])$ is a piecewise linear arc, it is clear that there exists an $\varepsilon_n$-chain cover $\mathcal{C}_{n+1}$ of $\Omega_{n+1}([-1,1])$ satisfying properties (1) and (2) of the Proposition.  Clearly, by taking the mesh of $\mathcal{C}_{n+1}$ sufficiently small, we may ensure that $\overline{\bigcup \mathcal{C}_{n+1}} \subset \bigcup \mathcal{C}_n$.
\end{proof}

\newtheorem*{thm ray}{Theorem~\ref{thm:ray}}
\begin{thm ray}
Suppose $Y$ is a continuum of the form $Y = X \cup R$, where $X$ is an arc-like continuum, $R$ is a ray, $X \cap R = \emptyset$, and $\overline{R} \setminus R \subseteq X$.  Then $Y$ can be embedded in the plane.
\end{thm ray}

\begin{proof}
Express $X$ as an inverse limit, $X = \varprojlim \langle [-1,1],f_m \rangle$.  Choose some point $x = \langle x_1,x_2,\ldots \rangle \in \overline{R} \setminus R$.  If $x_m = \pm 1$ for infinitely many $m$, then $x$ is an endpoint of $X$, meaning that for each $\varepsilon > 0$, there is an $\varepsilon$-chain cover $\{U_1,\ldots,U_k\}$ of $X$ such that $x \in U_1$.  Indeed, such a cover can be constructed by taking sufficiently large $m$ for which $x_m = \pm 1$, taking a sufficiently small mesh chain cover $\{I_1,\ldots,I_k\}$ of $[-1,1]$ by open intervals, and letting $U_j = \pi_m^{-1}(I_j)$ for each $j \in \{1,\ldots,k\}$, where $\pi_m \colon X \to [-1,1]$ is the $m$-th coordinate projection.  It then follows from Lemma~\ref{lem:chain} that $Y$ is arc-like, and therefore embeddable in the plane \cite{bing1951}.

Suppose, then, that $x_m = \pm 1$ for only finitely many $m$.  Applying Proposition~\ref{prop:plane}, we get an embedding $e \colon X \to \mathbb{H}$ with $e(x) = (0,0)$ and such that for any $\varepsilon > 0$, there exists an $\varepsilon$-chain cover $\mathcal{C}$ of $e(X)$ by round disks in $\mathbb{R}^2$ which has a unique element containing $(0,0)$.  Recursively, we may choose, for each $n \geq 1$, a $\frac{1}{n}$-chain cover $\mathcal{C}_n = \{U_1^n,\ldots,U_{k(n)}^n\}$ of $e(X)$ such that
\begin{enumerate}
\item each element $U \in \mathcal{C}_n$ is a round disk;
\item for each $U \in \mathcal{C}_{n+1}$, there exists $V \in \mathcal{C}_n$ such that $\overline{U} \subset V$; and
\item there exists a unique link $L_n \in \mathcal{C}_n$ containing $(0,0)$.
\end{enumerate}

For each $n \geq 1$ and each element $U \in \mathcal{C}_n$, consider the set $e^{-1}(U \cap e(X))$.  This is an open subset of $X$, as a subspace of $Y$, and the set $\{e^{-1}(U \cap e(X)): U \in \mathcal{C}_n\}$ is a chain cover of $X$ by open subsets of $X$.  By enlarging these sets slightly in $Y$, we may construct corresponding sets $U'$ which are open in $Y$, such that, for each $n \geq 1$:
\begin{enumerate}[label=(\arabic{*}$'$)]
\item $U' \cap X = e^{-1}(U \cap e(X))$ for each $U \in \mathcal{C}_n$;
\item for each $i,j \in \{1,\ldots,k(n)\}$, $(U_j^n)' \cap (U_i^n)' \neq \emptyset$ if and only if $|i-j| \leq 1$, so that the set $\mathcal{C}_n' = \{(U_1^n)',\ldots,(U_{k(n)}^n)'\}$ is a chain cover of $X$ in $Y$; and
\item for each $U \in \mathcal{C}_{n+1}$ and each $V \in \mathcal{C}_n$, if $\overline{U} \subset V$ then $\overline{U'} \subset V'$.
\end{enumerate}
It is straightforward to construct such collections $\{U': U \in \mathcal{C}_n\}$ by induction on $n$.  In particular, property (2$'$) is easy to obtain since we are working with (finite) chains of open sets in the metric space $Y$; alternatively, one could also appeal to more general theory for (2$'$), such as \cite[Theorem \S 21.XI.2]{kuratowski1966}.

Let $h'$ be a homeomorphism of $[0,\infty)$ onto $R$.  Choose points in the interval $[0,\infty)$ as follows.  Let $t_0 = 0$ and, for each $n \geq 1$, let $t_n \in (t_{n-1}, \infty)$ be such that $h'(t_n) \in ({L}_n)'$ and $h'([t_n,\infty)) \subset \bigcup \mathcal{C}_n'$.  Note that $t_n \to \infty$ as $n \to \infty$.  For each interval $[t_n,t_{n+1}]$, choose a partition $t_n = r_0^n < r_1^n < \cdots < r_{m(n)}^n = t_{n+1}$, fine enough so that the image of each subinterval under $h'$ is contained in an element of $\mathcal{C}_n'$; that is, for each $1 \leq \ell \leq m(n)$, there exists $1 \leq j(n,\ell) \leq k(n)$ with $h'([r_{\ell-1}^n,r_\ell^n]) \subset (U_{j(n,\ell)}^n)'$.

We now complete our construction of a plane embedding of $Y$, by constructing an embedding of the ray $R$ which follows the same pattern through the links $U_j^n$ in $\mathbb{R}^2$ as $R$ does through the links $(U_j^n)'$ in $Y$.  Choose an embedding $h \colon [0,\infty) \to \mathbb{R}^2$ with the following properties:
\begin{itemize}
\item $h(t_n) \in \partial L_n \setminus (\mathbb{H} \cup \bigcup \mathcal{C}_n)$ for each $n \geq 1$;
\item $h {\restriction}_{[0,t_1]}$ is a homeomorphism onto an arc in $\mathbb{R}^2 \setminus \bigcup \mathcal{C}_1$, and;
\item for each $n \geq 1$, $h {\restriction}_{[t_n,t_{n+1}]}$ is a homeomorphism onto an arc in
\[ \left( \bigcup \mathcal{C}_n \setminus \bigcup \mathcal{C}_{n+1} \right) \cup \{h(t_n)\} \]
such that, for each $1 \leq \ell \leq m(n)$, $h([r_{\ell-1}^n,r_\ell^n]) \subset U_{j(n,\ell)}^n$.
\end{itemize}

Refer to Figure~\ref{fig:ray embedding} for an illustration of the image of such a homeomorphism $h$ on a sample segment $[t_n,t_{n+1}]$.

\begin{figure}
\begin{center}
\includegraphics{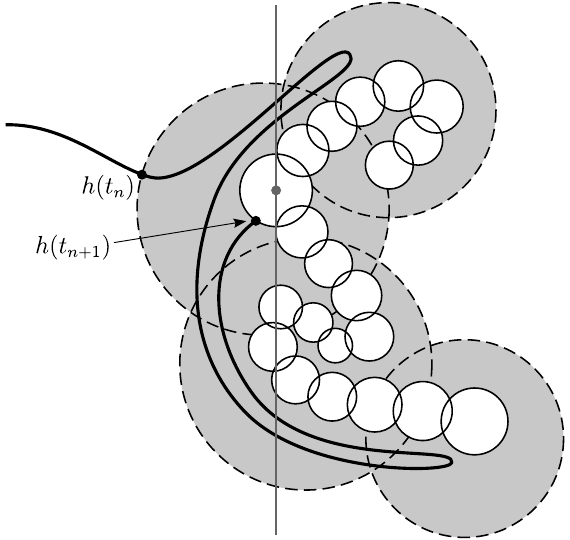}
\end{center}

\caption{An example illustration of two chain covers, $\mathcal{C}_n$ (in grey) and $\mathcal{C}_{n+1}$ (in white), of $X$, together with a portion of the embedding $h$ of the ray $[0,\infty)$ in $\mathbb{R}^2$, particularly the portion $h {\restriction}_{[t_n,t_{n+1}]}$.  The vertical line represents the $y$-axis $\partial \mathbb{H}$, with the origin marked with a grey dot.  Supposing that the links of $\mathcal{C}_n$ are enumerated, from top to bottom, $U^n_1,U^n_2,U^n_3,U^n_4$, so that $L_n = U^n_2$, in this example it is assumed that the pattern made by $h'$ on $[t_n,t_{n+1}]$ is as follows: $m(n) = 7$, and the numbers $j(n,\ell)$, for $1 \leq \ell \leq 7$, are $2,1,2,3,4,3,2$.}
\label{fig:ray embedding}
\end{figure}

Clearly, $h$ is a homeomorphism onto a ray $h([0,\infty))$ in the plane which does not intersect $e(X)$.  We claim that the function $\phi \colon Y \to \mathbb{R}^2$ given by $\phi {\restriction}_R = h \circ (h')^ {-1}$ and $\phi {\restriction}_X = e$ is a homeomorphism onto $h([0,\infty)) \cup e(X)$.  To prove this, it suffices to show that if a sequence $\langle h'(p_i) \rangle \subset R$ converges to a point $x \in X$, then $\langle h(p_i) \rangle$ converges to $e(x)$.

Let $n \geq 1$.  Let $\left( U_{j_x}^n \right)' \in \mathcal{C}_n'$ be an element which contains $x$.  Since $h'(p_i) \to x$, there exists an integer $i(n) \geq 1$ such that $h'(p_i) \in \left( U_{j_x}^n \right)'$ for all $i \geq i(n)$.  Let $i \geq i(n)$, and let $k \geq n$ and $1 \leq \ell \leq m(k)$ be such that $p_i \in [r_{\ell-1}^k,r_\ell^k]$, so that $h'(p_i) \in (U_{j(k,\ell)}^k)'$.  By properties (2) and (3$'$) above, there exists $1 \leq j \leq k(n)$ such that $U_{j(k,\ell)}^k \subset U_j^n$ and $(U_{j(k,\ell)}^k)' \subset (U_j^n)'$.  Thus $h'(p_i) \in \left( U_{j_x}^n \right)' \cap (U_j^n)'$ and, by property (2$'$), it follows that $|j-j_x| \leq 1$ and $U_{j_x}^n \cap U_j^n$ is nonempty as well.  According to the construction of $h$, we have $h(p_i) \in U_{j(k,\ell)}^k$, and so $h(p_i) \in U_j^n$.  Since $e(x) \in U_{j_x}^n$ and $U_{j_x}^n \cap U_j^n \neq \emptyset$, we conclude that the distance between $h(p_i)$ and $e(x)$ is less than $\frac{2}{n}$.  Thus, the sequence $\langle h(p_i) \rangle$ converges to $e(x)$, as desired.
\end{proof}

\newtheorem*{thm line}{Theorem~\ref{thm:line}}
\begin{thm line}
If $X$ is an arc-like continuum, then any compactification of a line with $X$ as remainder can be embedded in the plane.
\end{thm line}

\begin{proof}
Let $Y$ be a continuum such that $Y = X \cup L$ where $L$ is a line and $\overline{L} \setminus L = X$.  Let $h \colon \mathbb{R} \to L$ be a homeomorphism, and let $R_1 = h((-\infty,0])$ and $R_2 = h([0,\infty))$.  Then $R_1$ and $R_2$ are rays such that $L = R_1 \cup R_2$, and so $X = (\overline{R_1} \setminus R_1) \cup (\overline{R_2} \setminus R_2)$.  From Lemma~\ref{lem:chain} we conclude that either $X \cup R_1$ is arc-like or $X \cup R_2$ is arc-like -- indeed, for any chain cover $\mathcal{C} = \{U_1,\ldots,U_k\}$ of $X$ by open subsets of $Y$, either $U_1 \cap \overline{R_1} \setminus R_1 \neq \emptyset$ or $U_1 \cap \overline{R_2} \setminus R_2 \neq \emptyset$.  Assume that $X \cup R_1$ is arc-like.

Now we can apply Theorem~\ref{thm:ray}, with $X \cup h((-\infty,-1])$ as our arc-like continuum and $h([1,\infty))$ as our ray, to obtain an embedding $e \colon X \cup L \setminus h((-1,1)) \to \mathbb{R}^2$.  The endpoints $e \circ h(-1)$ and $e \circ h(1)$ are accessible points of $\mathrm{Im}(e)$, because the domain $\mathbb{R}^2 \setminus \mathrm{Im}(e)$ is locally connected at these points (see, e.g., \cite[Theorem~14.4]{newman1951}); alternatively, we could just as well have embedded $X \cup L \setminus h((-\frac{1}{2},\frac{1}{2}))$ in the plane -- an embedding of our initial continuum with arcs appended at the endpoints.  Moreover, since $\mathrm{Im}(e)$ is a tree-like continuum and, thus, does not separate the plane, we can draw an arc $A$ in the plane such that $A \cap \mathrm{Im}(e) = \{\textrm{endpoints of } A\} = \{e \circ h(-1),e \circ h(1)\}$.  Clearly $A \cup \mathrm{Im}(e) \subset \mathbb{R}^2$ is homeomorphic to $Y$.
\end{proof}

\bibliographystyle{amsplain}
\bibliography{CompactificationsR}

\end{document}